\documentclass[11 pt]{article}

\usepackage{graphics}
\usepackage{latexsym}
\usepackage{amsfonts}
\usepackage{amsmath}
\usepackage{amsthm}
\usepackage{multicol}
\usepackage{lscape}
\theoremstyle{plain}
\newtheorem{thm}{Theorem}
\newtheorem{lemma}[thm]{Lemma}
\newtheorem{conj}[thm]{Conjecture}

\newcommand{\F}{\mathcal{F}}
\newcommand{\M}{\mathcal{M}}
\newcommand{\B}{\mathcal{B}}
\newcommand{\N}{\mathbb{N}}
\newcommand{\R}{\mathbb{R}}
\newcommand{\C}{\mathcal{C}}
\DeclareMathOperator{\soc}{soc}
\DeclareMathOperator{\lk}{lk}

\begin{document}

\title{$f$-Vectors of Triangulated Balls}
\author{Samuel Kolins\\
Department of Mathematics,\\
Cornell University, Ithaca NY, 14853-4201, USA,\\
skolins@math.cornell.edu}

\maketitle

\begin{abstract}
We describe two methods for showing that a vector can not be the $f$-vector of a homology $d$-ball.
As a consequence, we disprove a conjectured characterization of the $f$-vectors of balls of dimension five and higher due to Billera and Lee.
We also provide a construction of triangulated balls with various $f$-vectors.
We show that this construction obtains all possible $f$-vectors of three and four dimensional balls and we conjecture that this result also extends to dimension five.
\end{abstract}

\section{Introduction}

A fundamental invariant of a simplicial complex is its collection of face numbers or $f$-vector.
A major area of study is understanding the possible $f$-vectors of various types of simplicial complexes.
In this paper we prove some new results on the $f$-vectors of simplicial complexes that are triangulations of balls.

A complete characterization of the $f$-vectors of simplicial polytopes was given in 1981 with the proof of the $g$-theorem by Billera and  Lee \cite{BilleraLee1} and Stanley \cite{Stanleyg}.
The $g$-conjecture asserts that this characterization also holds in the more general setting of all triangulated spheres.
In \cite{BilleraLee2}, Billera and Lee calculate a set of conditions on the $f$-vectors of triangulated balls that would follow from the $g$-conjecture.
Billera and Lee conjecture that these conditions are not only necessary but also sufficient for a characterization of the $f$-vectors of balls.
Recently, Lee and Schmidt confirmed this conjecture for three and four dimensional balls \cite{LeeSchmidt}.

In this paper we present two methods that show that certain vectors are not the $f$-vectors of triangulated balls.
As a consequence, we show that the Billera and Lee conditions are not sufficient in dimensions five and higher.
In both approaches, we assume that a ball with the a certain $f$-vector exists and then show that there must be some way to split the ball along a co-dimension one face to create two new balls.
For some $f$-vectors we can show that the new balls created by this splitting can not exist.
The first technique relates one of the Betti numbers of the face ring of the ball to the existence of a co-dimension one face along which we can split the ball.
This has the advantage of being relatively straightforward to compute in particular examples.
In the second method we look at all possible one skeletons of a ball with a given $f$-vector and show that in each case the desired type of splitting is possible.
This is used to generate an infinite class of counterexamples to the Billera and Lee conjecture in every dimension greater than four. 

The second portion of the paper presents a construction of balls with prescribed $f$-vectors.
In dimensions three and four this result duplicates the work of Lee and Schmidt in obtaining all possible $f$-vectors of balls.
For the dimension five case, we conjecture that this construction gives all possible $f$-vectors.
However, in dimensions higher than five not all $f$-vectors of balls can be obtained with this approach.

The structure of this paper is as follows.
In section \ref{S:Background} we review the needed background material.
In section \ref{S:Known} we discuss previously known and conjectured conditions on the $f$-vectors of balls.
In section \ref{S:nec} we present our methods for creating new necessary conditions on $f$-vectors.
In section \ref{S:Const} we give our construction and in section \ref{S:Sum} we discuss some consequences of the construction as well as our conjecture for the $f$-vectors of five dimensional balls.

\section{Notation and Background}{\label{S:Background}}

We begin by discussing some needed background on simplicial complexes, the face ring, and commutative algebra.
Stanley's book \cite{StanleyCCA} is a good reference for most of the material in this section.

A \emph{simplicial complex} $\Delta$ on the vertex set $[n] = \{1,2,\ldots , n\}$ is a collection of subsets of $[n]$ that is closed under inclusion and contains all of the one element sets $\{i\}$ for $i \in [n]$.
The elements of $\Delta$ are called \emph{faces}, and the \emph{dimension of a face} $F \in \Delta$ is $\dim F := |F|-1$.
The \emph{dimension of $\Delta$} is equal to the maximum of the dimensions of all of its faces.
A \emph{facet} of $\Delta$ is any maximal face with respect to inclusion.
A simplicial complex is \emph{pure} if all of its facets have the same dimension.

If $F \in \Delta$ then the \emph{ link of $F$ in $\Delta$} is $\lk_{\Delta} (F) := \{ G \in \Delta : G \cup F \in \Delta, G \cap F = \emptyset \}$.
For $W \subset [n]$, $\Delta_W := \{ \sigma \in \Delta : \sigma \subset W \}$ is the \emph{induced sub-complex} on the vertex set $W$.

Let $\Delta$ be a simplicial complex of dimension $(d-1)$.
The \emph{face numbers} of $\Delta$, denoted by $f_i (\Delta)$, are the number of $i$-dimensional faces of $\Delta$.
When it is clear to what complex we are referring we will often just write $f_i$ for the face numbers. 
So $f_{-1} = 1$ (corresponding to the empty set), $f_{0} = n$, and $f_i = 0$ for $i \geq d$.
The \emph{$f$-vector} of $\Delta$ is the list of the face numbers, $f(\Delta) := (f_{-1}, f_0, f_1, \ldots , f_{d-1})$.
The \emph{$h$-vector} of $\Delta$, $h(\Delta) = (h_0,h_1, \ldots , h_d)$ contains the same combinatorial information as the $f$-vector but is often easier to use.
Its entries are defined from the face numbers by 
\begin{equation*}
\sum_{i=0}^d h_i x^i = \sum_{i=0}^{d} f_{i-1} x^i (1-x)^{d-i}
\end{equation*}
Define $g_i (\Delta ) := h_i(\Delta) - h_{i-1} (\Delta)$.

Let $\{e_1, e_2, \ldots , e_n\}$ be the standard basis of $\R^n$.
The \emph{geometric realization} $|\Delta|$ of $\Delta$ is the union of the convex hull of $\{e_{i_1}, \ldots , e_{i_k}\}$ over all faces $\{i_1, \ldots , i_k\}$ of $\Delta$.
We say $\Delta$ is \emph{homeomorphic} to a topological space $X$ if $|\Delta|$ is homeomorphic to $X$.
A \emph{triangulation} of a topological space $X$ is a simplicial complex that is homeomorphic to $X$.

In some of our work, instead of considering triangulated balls it will be useful to consider the larger class of homology balls.
All of our homology will be taken with coefficients in the integers.
A pure simplicial complex $\Delta$ of dimension $(d-1)$ is a \emph{homology $(d-1)$-manifold} if for every non-empty face $F \in \Delta$ the link of $F$ has the same homology as the $(d-1-|F|)$-sphere or $(d-1-|F|)$-ball.
The \emph{boundary} of a homology $(d-1)$-manifold $\Delta$ is defined to be $\partial \Delta := \{ F \in \Delta | H_{d-1-|F|} (F) = 0\}$.
From \cite{Mitchell90} we know that the boundary of a homology $(d-1)$-manifold is either empty or a homology $(d-2)$-manifold without boundary.
A \emph{homology $(d-1)$-sphere} is a homology $(d-1)$-manifold with empty boundary and the same homology as the $(d-1)$-sphere.
A \emph{homology $(d-1)$-ball} is an homology $(d-1)$-manifold with the same homology as the $(d-1)$-ball and boundary a homology $(d-2)$-sphere.
From the long exact sequence of the homology of the pair $(\Delta, \partial \Delta)$, for a homology $(d-1)$-ball $\Delta$ we have $H_{d-1} (\Delta, \partial \Delta) = \mathbb{Z}$.

Note that for a $(d-1)$ dimensional simplicial complex, $h_d = \sum_{i=0}^d (-1)^{d-i} f_{i-1} = (-1)^{d-1} \tilde\chi (|\Delta|)$.
In particular, if $\Delta$ is a homology $(d-1)$-ball then $h_d = 0$.

A \emph{shelling order} of a pure simplicial complex $\Delta$ is an ordering $\{F_1, \ldots , F_{f_{d-1}} \}$ of the facets of $\Delta$ such that for $j=1, \ldots , f_{d-1}$, when $F_j$ is added to $\cup_{i=1}^{j-1} F_i$ there is a unique new face that is minimal with respect to inclusion.
This face is called the \emph{restriction face} and is denoted $r(F_j)$.
A complex is \emph{shellable} if there exists a shelling order of its facets.
We can obtain the $h$-vector of a shellable complex from the restriction faces of any shelling order by $h_i = |\{F_j : |r(F_j)|=i \}|$ \cite[Proposition 2.3]{StanleyCCA}.

Let $k$ be an infinite field of arbitrary characteristic.
Define $R:=k[x_1, \ldots, x_n]$.
For a simplicial complex $\Delta$ on vertex set $[n]$ the \emph{face ring} (or \emph{Stanley-Reisner ring}) is

\begin{equation*}
k[\Delta] := R/I_{\Delta}, \ \ \ \ \ \mbox{where  }  I_{\Delta} = (x_{i_1} x_{i_2} \cdots x_{i_k} : i_1 < i_2 < \cdots < i_k, \{i_1, i_2, \cdots , i_k\} \notin \Delta).
\end{equation*}

A simplicial complex $\Delta$ is called \emph{Cohen-Macaulay} if $k[\Delta]$ is Cohen-Macaulay.
In \cite{Reisner}, Reisner gives a characterization of Cohen-Macaulay complexes in terms of the homology of the links of the faces in the complex.
As a consequence of this result, all homology balls and spheres are Cohen-Macaulay.

A \emph{linear system of parameters} (l.s.o.p) for $k[\Delta]$ is a collection of degree one elements $\theta_1, \ldots , \theta_d \in k[\Delta]$ such that $k[\Delta]/(\theta_1, \ldots , \theta_d)$ is finite dimensional over $k$.
For an infinite field $k$ a generic choice of $d$ degree one elements of $k[\Delta]$ is a l.s.o.p.
Define $k(\Delta) := k[\Delta]/(\theta_1, \ldots , \theta_d)$ for any l.s.o.p $\theta_1, \ldots , \theta_d$.

Let $T$ be a graded ring such that $T=R/I$ for some ideal $I$. 
Denote by $T_i$ the $i$th homogeneous component of $T$.
The \emph{Hilbert function} of T is given by $F(T,i) :=  \dim_k T_i$.
For a Cohen-Macaulay complex $F(k(\Delta),i) = h_i$.

Let $S$ be one of the rings $k[\Delta]$ or $k(\Delta)$.
Thinking of $S$ as an $R$-module the minimal free resolution of $S$ has form
\begin{equation*}
 0 \rightarrow \bigoplus_j S[-j]^{\beta_{l,j}} 
 \rightarrow \bigoplus_j S[-j]^{\beta_{l-1,j}}  \rightarrow \cdots
  \rightarrow \bigoplus_j S[-j]^{\beta_{1,j}}
  \rightarrow \bigoplus_j S[-j]^{\beta_{0,j}}
  \rightarrow  S \rightarrow 0
\end{equation*} 
Here $S[-j]$ is the module $S$ shifted by degree $j$ and $l$ is the length of the resolution, also called the \emph{homological dimension} of $S$.
The $\beta_{i,j}$ are called the \emph{Betti numbers} of $S$.
Using the Auslander-Buchsbaum formula \cite[Theorem I.11.2]{StanleyCCA}, for a Cohen-Macaulay complex $\Delta$ the minimal resolution of $k[\Delta]$ has length $n-d$ and the minimal resolution of $k(\Delta)$ has length $n$ (where $n$ is the number of vertices in $\Delta$ and hence also the number of variables in $R$).
The Betti numbers of $k[\Delta]$ are related to the topology of the complex $\Delta$.
One powerful expression of this relationship is Hochster's Formula \cite{Hochster}.

\begin{equation}{\label{Hochster}}
 \beta_{i,j} (k[\Delta]) = \sum_{W \subset V, |W|=j} \dim_k (\tilde{H}_{j-i-1} \, (\Delta_W; k))
\end{equation}

Finally, we turn our attention to terminology related to monomial ideals.
A non-empty set $M$ of monomials is an \emph{order ideal} if for all $m \in M$ and $m' | m$ we have $m' \in M$.
Denote by $M_i$ the monomials in $M$ of degree $i$.
The \emph{degree sequence} of $M$ is the vector $(|M_0|,|M_1|,|M_2|, \ldots )$.
A vector that is the degree sequence of some order ideal of monomials is called an \emph{M-vector}.

Given a monomial $m=X_1^{a_1} X_2^{a_2} \cdots X_n^{a_n}$ in the indeterminates $X_1, X_2, \ldots, X_n$ define the \emph{degree} of $m$ to be $\sum_{i=1}^n a_i$.
For monomials of the same degree define the \emph{lexicographic order} $<_l$ by $X_1^{a_1} X_2^{a_2} \cdots X_m^{a_m} <_l X_1^{b_1} X_2^{b_2} \cdots X_m^{b_m}$ if there exists some $k \in \{1,2,\ldots m\}$ such that $a_i = b_i$ for $i < k$ and $a_k > b_k$.
An order ideal $M$ is called a \emph{lex ideal} if for each $j$ the elements of $M_j$ are the first $|M_j|$ monomials of degree $j$ in the lexicographic order.

Define the \emph{reverse-lexicographic order} (or rev-lex order) $<_{r l}$ on monomials by $X_1^{a_1} X_2^{a_2} \cdots X_m^{a_m} <_{rl} X_1^{b_1} X_2^{b_2} \cdots X_m^{b_m}$ if there exists some $k \in \{1,2,\ldots m\}$ such that $a_i = b_i$ for $i > k$ and $a_k < b_k$.
An order ideal $M$ is \emph{compressed} if for each $j$ the elements of $M_j$ are the first $|M_j|$ monomials of degree $j$ in the rev-lex order.
Given any $M$-vector $h$ there exists a compressed order ideal of monomials with degree sequence equal to $h$ \cite[Proposition 1]{BilleraLee1}.

Let $m$ be a monomial and $c \in \N$ such that the degree of the $m$ is less than or equal to $c$.
There is a unique way to write $m$ as $m = X_{e_1} X_{e_2} \cdots X_{e_c}$ where $ 0 \leq e_1 \leq e_2 \leq \cdots \leq e_c$ and we take $X_0 = 1$.
This is called the \emph{extended representation} of $m$.
Define the \emph{partial order} $<_p$ on monomials of degree less than or equal to $c$ by $ X_{e_1} X_{e_2} \cdots X_{e_c} \leq_p  X_{e_1'} X_{e_2'} \cdots X_{e_c'}$ if and only if $e_k \leq e_k'$ for $k=1, \ldots , c$.
Note that any initial segment of monomials in the rev-lex order is also an initial segment in this partial order. 
For $C \in \N$ we also define a partial order $<_p$ on $C$-subsets of the natural numbers.
If $S = \{i_1, \ldots, i_C\}$ and $T = \{j_1, \ldots , j_C\}$ are $C$- subsets of $\N$ with elements listed in increasing order 
then $S \leq_p T$ if and only if $i_k \leq j_k$ for $k = 1, \ldots, C$. 

In addition to the definition of an $M$-vector in terms of order ideals there is also a numerical characterization of $M$-vectors due to Macaulay \cite{Macaulay}.
Given any $l,i \in \N$ there is a unique expansion of $l$ of the form
$$l = \binom{n_i}{i} + \binom{n_{i-1}}{i-1} + \cdots + \binom{n_j}{j},  \ \ \ \ \ n_i > n_{i-1} > \cdots > n_j \geq j \geq 1.$$
This is called the \emph{$i$-canonical representation of $l$}.
Define $l^{<i>}$, the \emph{$i$th pseudo-power of $l$}, by
$$l^{<i>} = \binom{n_i+1}{i+1} + \binom{n_{i-1}+1}{i} + \cdots + \binom{n_j+1}{j+1}$$
$(h_0, h_1, \ldots)$ is an $M$-vector if and only if $h_0=1$ and $0 \leq h_{i+1} \leq h_i^{<i>}$ for $i \geq 1$.

Another characterization of $M$-vectors is in terms of Cohen-Macaulay complexes.
An integer vector $h = (h_0, h_1, \ldots, h_d)$ is an $M$-vector if and only if there exists a $(d-1)$-dimensional Cohen-Macaulay complex $\Delta$ such that $h (\Delta) = h$.
In fact, this result is also true for the more restrictive class of shellable complexes \cite[Theorem 3.3]{StanleyCCA}.

\section{Known and Conjectured Necessary Conditions on the $h$-vectors of Homology Balls}{\label{S:Known}}

In this section we discuss some of the previously known necessary conditions on the $h$-vectors of homology balls as well as some conjectured conditions on these $h$-vectors.
Many of the results are obtained by examining the relationship between the $h$-vector of a homology ball and the $h$-vector of its boundary homology sphere.
We use the conditions that the $g$-conjecture places on the $h$-vector of the boundary sphere to obtain possible restrictions on the $h$-vector of the original ball.  

\begin{conj}
(The $g$-Conjecture)
An integer vector $(h_0,h_1,\ldots,h_{d})$ with $h_0=1$ is the $h$-vector of a homology $d$-sphere if and only if
\begin{enumerate}
\item $h_i = h_{d-i}$ for $0 \leq i \leq \lfloor d/2 \rfloor$
\item $(1,g_1,g_2, \ldots , g_{\lfloor d/2 \rfloor})$ is an $M$-vector, where $g_i = h_i-h_{i-1}$.
\end{enumerate} 
\end{conj}

The $g$-conjecture is not known to hold for all homology spheres (or all triangulated spheres) but has been proved for boundaries of simplicial polytopes \cite{BilleraLee1, Stanleyg}.
The relations of condition one of the $g$-conjecture are called the Dehn-Sommerville equations and are known to hold for all homology spheres \cite{Klee64}.
Additionally, Barnette proved that for all triangulated spheres the initial part $(g_0,g_1, g_2)$ of the $g$-vector is an $M$-vector \cite{Barnette73} (in fact this result is true for the much larger class of all doubly Cohen-Macaulay complexes, which include all homology spheres \cite{Nevo08}).

In order to use these known and conjectured conditions on the $h$-vectors of homology spheres we derive a relationship between the $h$-vector of a homology ball and the $h$-vector of its boundary.
In \cite{Macdonald71}, MacDonald proves a generalization of the Dehn-Sommerville equations for triangulated manifolds with possibly non-empty boundary.
Here we will use the equivalent result for homology manifolds expressed in terms of $h$-vectors \cite[Theorem 3.1]{Novik09}.

\begin{thm}{\label{GDS}}
Let $\Delta$ be a $(d-1)$-dimensional homology manifold with boundary.
Then
$$h_{d-i}(\Delta) - h_i(\Delta) = \binom{d}{i} (-1)^{d-1-i} \tilde\chi (|\Delta|) - g_i(\partial \Delta)$$
for all $0 \leq i \leq d$.
\end{thm}
  
In the case where $\Delta$ is a $(d-1)$ dimensional homology ball, this reduces to $h_{i}(\Delta) - h_{d-i}(\Delta) = g_i(\partial \Delta)$.
Let $\Delta_k$ be the cone over $\Delta$ taken $k$ times.
Then $\Delta_k$ is a $(d-1+k)$ dimensional homology ball with boundary a $(d-2+k)$ dimensional homology sphere.
Following the argument of Billera and Lee \cite[Corollary 3.14]{BilleraLee2} yields 
$g_i (\partial \Delta_k) = h_i (\Delta_k) - h_{d+k-i} (\Delta_k) = h_i (\Delta) - h_{d+k-i} (\Delta)$, $0 \leq i \leq d+k$.
Combining this result with the $g$-conjecture yields the following set of conditions.

\begin{conj}{\label{gconditions}}
If $(h_0, \ldots, h_d)$ is the $h$-vector of homology $(d-1)$-ball and we take $h_i=0$ for $i>d$ then 
$(h_0-h_{d+k}, h_1 - h_{d+k-1}, \ldots , h_m-h_{d+k-m})$ is an $M$-vector for $k=0, \ldots ,d+1$, $m= \lfloor (d+k-1)/2 \rfloor$.
\end{conj}

Even though the $g$-conjecture has not been settled for the case of homology spheres, many of the conditions in Conjecture \ref{gconditions} can be verified.

The case $k=d+1$ is the statement that the original $h$-vector of the ball, $(h_0, \ldots, h_d)$, is an $M$-vector.
Since homology balls are Cohen-Macaulay complexes their $h$-vectors are $M$-vectors.

Next consider the cases $d-3 \leq k \leq d$.
For $k > d-3$, since $h_d (\Delta) =0$ for any homology ball $\Delta$ Conjecture \ref{gconditions} reduces to the fact that the $h$-vector of $\Delta$ is an $M$-vector, which was discussed above.
For the case $k=d-3$ we must show that $(h_0, h_1, \ldots, h_{d-3}, h_{d-2}-h_{d-1})$ is an $M$-vector.
Since we already know that the $h$-vector of $\Delta$ is an $M$-vector, we only need to show that $h_{d-2} - h_{d-1}$ is non-negative.
As a consequence of \cite[Corollary 4.29]{Swartz09} $h_{d-2} \geq h_{d-1}$ for any homology ball, giving the desired inequality.

Finally, from Barnette's result the initial segments $(h_0-h_{d+k}, h_1 - h_{d+k-1}, h_2-j_{d+k-2})$ of the vectors in Conjecture \ref{gconditions} are $M$-vectors.

Combining these known results shows that all of the conditions in Conjecture \ref{gconditions} are valid for homology balls of dimension less than or equal to four.

\section{A New Type of Necessary Condition}{\label{S:nec}}

In \cite[Conjecture 5.1]{BilleraLee2}, Billera and Lee conjecture that the conditions of Conjecture \ref{gconditions} are also sufficient for the $h$-vectors of triangulated balls (in fact, they conjecture that the conditions are sufficient for the more restrictive collection of polyhedral balls).
This conjecture was verified for balls of dimension three and four by Lee and Schmidt \cite{LeeSchmidt}.
In this section we disprove the sufficiency of the conditions in Conjecture \ref{gconditions} for homology balls of dimension five and higher.

\subsection{The Betti Diagram of the Face Ring}{\label{S:Betti}}

Our first goal will be to understand the upper right entry of the Betti diagram of the face ring modulo a linear system of parameters, $\beta_{n,n+1} (k(\Delta))$.
From Hochster's Formula (\ref{Hochster})

\begin{equation}{\label{socle}}
 \beta_{n-d,n-d+1} (k[\Delta]) = \sum_{W \subset [n], |W|=n-d+1} \dim_k (\tilde{H}_{0} \, (\Delta_W, k))
\end{equation}

From \cite[Lemma 3.7]{Hibi94}, the group $\tilde{H}_{0} \, (\Delta_W, k)$ in equation (\ref{socle}) is trivial whenever $[n] \backslash  W$ is not a face of $\Delta$.
Therefore, $\beta_{n-d,n-d+1} (k[\Delta])$ is non-zero if and only if there exists a $(d-2)$-face of $\Delta$ whose removal disconnects the complex.
If $\Delta$ is a homology ball then Mayer-Vietoris sequences show that the two pieces created by removing this $(d-2)$-face are also homology balls.

Further, $\beta_{n-d,n-d+1} (k[\Delta])$ is equal to the dimension of the degree one portion the socle of $k(\Delta)$ \cite[Theorem 12.4]{StanleyCCA}.
Since $k(\Delta)$ has homological dimension $n$, the dimension of the degree $i$ portion of the socle of $k(\Delta)$ is given by the Betti number $\beta_{n,n+i} (k(\Delta))$.
Combining these facts gives
\begin{equation*}{\label{E:Betti}}
\beta_{n-d,n-d+1} (k[\Delta]) = \dim_k (\soc k(\Delta))_1 = \beta_{n,n+1} (k(\Delta))
\end{equation*}

Since the Hilbert Function of $k(\Delta)$ is equal to the $h$-vector of $\Delta$, we can use Peeva's cancellation technique  to relate the $h$-vector of $\Delta$ to the Betti numbers of $k(\Delta)$.
Let $L$ be the lexicographic ideal such that the Hilbert function of $R/(L)$ is equal to the $h$-vector of $\Delta$.  By \cite[Theorem 1.1]{Peeva}, the Betti number $\beta_{n,n+1} (k(\Delta))$ is bounded above by $\beta_{n,n+1} (R/(L))$ and below by $\beta_{n,n+1} (R/(L)) - \beta_{n-1,n+1} (R/(L))$.
Therefore, when $\beta_{n,n+1} (R/(L)) - \beta_{n-1,n+1} (R/(L)) > 0$ there exists a co-dimension one face that divides our homology ball into two separate homology $(d-1)$-balls.

\subsection{Splitting Balls}{\label{S:glue}}{\label{S:split}}

Next we investigate the effect on the $h$-vector when a homology ball is split along a single co-dimension one face.
Let $\Delta_1$ and $\Delta_2$ be two homology $(d-1)$-balls that can be joined along a common homology $(d-2)$-ball $B$ to form a single homology $(d-1)$ ball $\Delta$.
On the level of $f$-vectors 
\begin{equation*}
f_i(\Delta_1) + f_i (\Delta_2) = f_i (\Delta) + f_i (B)
\end{equation*}
A straightforward calculation shows

\begin{equation}{\label{E:combine}}
h_i(\Delta_1) + h_i (\Delta_2) = h_i (\Delta) + (h_i (B) - h_{i-1} (B) )
\end{equation}
where we take $h_{-1} (B) = h_d (B) = 0$.
As a special case of equation (\ref{E:combine}), when we join two homology balls along a single co-dimension one face the $h$-vector of the resulting complex is the sum of the $h$-vectors of the two component homology balls but with $h_0$ reduced by one (causing $h_0$ of the resulting complex to still equal one) and $h_1$ increased by one.

With these tools, we now consider the $h$-vector $(1,4,5,7,3,2,0)$.
Assume there is a homology five-ball $\Delta$ with this $h$-vector.
Using the result of Eliahou and Kervaire \cite[section 3]{EliahouKervaire90}, we calculate the Betti numbers of the $R/(L)$, where $L$ is the lexicographic ideal such that the Hilbert function of $R/(L)$ is equal to $(1,4,5,7,3,2,0)$.
This yields $\beta_{n,n+1} (R/(L)) =1 $ and $\beta_{n-1,n+1} (R/(L)) = 0$, 
which by Peeva's results implies $\beta_{n,n+1} (k(\Delta)) =1$.
Therefore there exists a co-dimension one face along which we can split $\Delta$.

Using our formula for the $h$-vector obtained by combining homology balls, we now look for possible $h$-vectors for the two   homology five-balls created when we split our original homology ball.
Each of the two smaller homology balls must satisfy all of the known portions of Conjecture \ref{gconditions} as discussed in section \ref{S:Known}.

The only two pairs of options for the $h_1$ values of our smaller homology balls are 2,1 or 3,0.
If we take 2 and 1 as the $h_1$ values the largest possible corresponding values of $h_2$ are 3 and 1, which do not sum to the $h_2$ value of our original ball.
If we take 3 and 0 as our $h_2$ values then one homology ball must have $h$-vector $(1,3,5,7,3,2,0)$.
However, the $g$-vector of the boundary homology four-sphere of this ball would be (1,1,2), violating a known portion of the $g$-conjecture.

We have shown that there must be a division of our homology five-ball into two smaller homology balls, but also that no such division exists.
Therefore $(1,4,5,7,3,2,0)$ is not the $h$-vector of a homology five-ball, even though it satisfies all of the conditions of Conjecture \ref{gconditions}.

Similar calculations show that there are other $h$-vectors that satisfy the conditions of Conjecture \ref{gconditions} but are not the $h$-vectors of homology balls.  
Examples include $(1,4,6,9,4,2,0)$ and $(1,5,6,8,4,3,0)$.

\subsection{A combinatorial approach}{\label{S:Comb}}

Some of the results of the previous section can also be obtained using a more combinatorial viewpoint.
Given an $h$-vector $(1,h_1, \ldots ,h_d)$ look at all of the possible graphs with $h$-vector $(1,h_1,h_2)$.
For each of these graphs count the maximum number of triangles possible in a simplicial complex with the given graph as its one-skeleton (or equivalently, the number of triangles in the flag complex or clique complex induced by the graph).
Compute the $h$-vector $(1,h_1,h_2,h_3')$ of the complex with all possible triangles.
Since removing triangles from a complex $\Delta$ decreases $h_3(\Delta)$ and adding faces of dimension greater than two does not change $h_3 (\Delta)$, any one skeleton of a complex with $h$-vector $(1,h_1, \ldots ,h_d)$ must have $h_3' \geq h_3$.

Doing an exhaustive search we find that in the case of the $h$-vector $(1,4,5,7,3,2,0)$ the only graphs that obtain $h_3' \geq h_3$ have a vertex of degree less than or equal to five.
In a homology ball each vertex must be contained in at least one facet.
This forces each vertex in a homology five-ball to have degree at least five.
If a vertex has degree exactly five then it is contained in only one facet.
In this case we can remove the facet to create a homology five-ball with $h$-vector $(1,3,5,7,3,2,0)$.
As argued in the previous section, this contradicts the known conditions of Conjecture \ref{gconditions}.
Therefore each vertex must have degree at least six, so no homology five-ball with $h$-vector $(1,4,5,7,3,2,0)$ exists.

Using this approach, we next describe a collection of $h$-vectors that satisfy all of the conditions of Conjecture \ref{gconditions} but such that the existence of a homology ball with one of these $h$-vectors would contradict Barnette's partial proof of the $g$-conjecture.
This will show that the conditions of Conjecture \ref{gconditions} are not sufficient in dimensions five and higher.

Let $x,y$ be integers with $x > 4$ and $1 < y < x$.
Consider the $h$-vector

$$\left (1, x , \binom{x}{2} , \binom{x+1}{3} - 2 ,  \binom{x+1}{3} - 2, \ldots ,
  \binom{x+1}{3} - 2 ,  \binom{x}{2} - \left( \binom{y}{2} +1 \right), x - y , 0 \right)$$

We claim that this satisfies all of the conditions in Conjecture \ref{gconditions}.

The case $k=0$ (the boundary sphere condition) requires that
$$\left( 1 , y , \binom{y}{2} + 1 \right)$$
is an $M$-vector, which follows since $\binom{y}{2} < \binom{y+1}{2}$.

For the case $k=1$ (the condition that comes from taking a cone) 
$$\left(1, x , \binom{x}{2} - x + y , \binom{x+1}{3} - \binom{x}{2} + \binom{y}{2} -1 \right)$$
must be an $M$-vector.

To see this, first note that $\binom{x}{2} - x + y = \binom{x}{2} - \binom{x-1}{1}  + y -1 = \binom{x-1}{2} + \binom{y-1}{1}$.

Therefore $\left( \binom{x}{2} - x + y \right)^{<2>} = \binom{x}{3} + \binom{y}{2}
= \binom{x}{3} + \binom{x}{2} - \binom{x}{2} + \binom{y}{2}
= \binom{x+1}{3} - \binom{x}{2} + \binom{y}{2}$.
Combining this with the fact that $x > y$ shows that the desired vector is an $M$-vector.

For the case $k=2$ 
$$\left(1, x , \binom{x}{2}, \binom{x+1}{3} - 2 - x + y, \binom{x+1}{3} - \binom{x}{2} + \binom{y}{2} - 1 \right)$$
must be an $M$-vector.

Since $x > y$, the step from the second to third entry satisfies Macaulay's condition.
Note that since $x > y >1$ and $x>4$, $\binom{x}{2} - \binom{y}{2} = \frac{1}{2} ( x(x-1) -y(y-1)) > \frac{1}{2} ( x(x-1) -y(x-1)) > x - y$.
Therefore, the step from the third to fourth entry is non-increasing.
All of the remaining checks of Macaulay's conditions and non-negativity needed to show that the desired vector is an $M$-vector are straightforward.

All of the higher $k$ values result in vectors of one of the forms
$$\left(1, x , \binom{x}{2}, \binom{x+1}{3} -2 , \ldots , \binom{x+1}{3} -2, \binom{x+1}{3} - 2 - x + y, \binom{x+1}{3} - \binom{x}{2} + \binom{y}{2} - 1 \right),$$
$$\left(1, x , \binom{x}{2}, \binom{x+1}{3} -2 , \ldots , \binom{x+1}{3} -2, \binom{x+1}{3} - 2 - x + y,  \binom{x}{2} - \binom{y}{2} - 1 \right),$$
$$\left(1, x , \binom{x}{2}, \binom{x+1}{3} -2 , \ldots , \binom{x+1}{3} -2,  \binom{x}{2} - x - \left(\binom{y}{2} -y \right) - 1 \right),$$
or the original $h$-vector itself.
Using the same arguments as in the previous cases, these are also all $M$-vectors.

Calculating the corresponding $f$-vector yields
\begin{equation*}
f_0 = d+x \ \ \ \ \ \ \ \ 
\binom{d+x}{2}-f_1 = x \ \ \ \ \ \ \ \ 
\binom{d+x}{3}-f_2 = \frac{x^2}{2} + \left( \frac{2d-3}{2} \right) x + 2
\end{equation*}

Let $\Delta$ be a homology $(d-1)$-ball with the above $h$-vector.
Call $(i+1)$-subsets of $[d+x]$ that are not elements of $\Delta$ absent $i$-faces.
Thus $\binom{d+x}{i+1} - f_i$ counts the number of absent $i$-faces of $\Delta$.

We claim that every vertex of $\Delta$ has degree at least $d$.
To see this, first note that any vertex of degree less than $(d-1)$ would not be contained in any facet of our complex.
If there was a vertex of degree $(d-1)$ this vertex would be contained in exactly one facet.
Removing this facet from our homology ball would decrease $h_1$ by one, leaving us with a homology ball whose boundary homology sphere would have $g$-vector $(1 , y -1 , \binom{y}{2} +1)$, contradicting Barnette's result.

Let $G$ be the graph of absent edges of $\Delta$.
The vertex set of $G$ is $[d+x]$, the same as the vertex set of $\Delta$, and $\{a,b\}$ is an edge of $G$ if and only if $\{a,b\} \notin \Delta$.
By the above claim the maximum degree of any vertex in $G$ is $x-1$.

For each edge $\{a,b\}$ of $G$ and each vertex $c \notin \{a,b\}$ the triangle $\{a,b,c\}$ is an absent triangle.
If $G$ has no vertex of degree at least two, then for each combination of an absent edge and a vertex not in that edge there is a distinct absent triangle.
This results in a total of at least $x (x+d-2)$ absent triangles, far more than the $\frac{x^2}{2} + \left( \frac{2d-3}{2} \right) x + 2$ allowed absent triangles.
We can therefore assume that $G$ has some vertex $v$ of degree $k$ where $2 \leq k <x$.

Label the edges of $G$ by $e_1, e_2, \ldots , e_x$ where $v$ is contained in $e_i$ for $1 \leq i \leq k$.
Let $G_k$ be the graph on vertex set $[x+d]$ with edges $\{e_j\}_{j=1}^k$.
Let $A_k := \{ \{a,b,c\} \subset [x+d] : a,b,c$  distinct and $G_k$ contains 
at least one of the edges  $\{a,b\}, \{b,c\}, \{a,c\}  \}$.
Then $|A_x|$ is less than or equal to the number of absent triangles in $\Delta$.
 
We now compare the sets $A_{i-1}$ and $A_i$ for $i=1, \ldots ,x$.
First note that $A_{i-1} \subseteq A_i$. 
When moving from $A_{i-1}$ to $A_i$ the new elements are those containing the two endpoints of $e_i$ and one other vertex which is not adjacent in $G_{i-1}$ to either endpoint of $e_i$.
For $i \leq k$ this implies $|A_i \backslash A_{i-1}| = x+d-1-i$.
For $i \geq k+1$ the graph $G_{i-1}$ has $i-(k+1)$ edges that do not contain $v$.
Therefore there are at least $(x+d-3-(i-(k+1))$ vertices that are not in $e_i$ and are not adjacent in $G_{i-1}$ to either endpoint of $e_i$.
Thus $A_i \backslash A_{i-1}$ contains at least $(x+d-3-(i-(k+1))$ elements.
In total, $|A_x|$ is bounded below by
\begin{align*}
& (x+d-2)+ (x+d-3) + \cdots + (x+d-1-k) + (x+d-3) + \cdots + (x+d-3-(x-(k+1)) \\
&= \left[ (x+d-2)+ (x+d-3) + \cdots + (d-1) \right] + (k-1)(x-k) \\
&= \frac{x^2}{2} + \frac{2d-3}{2} \cdot x + (k-1)(x-k)
\end{align*}
Since $x>4$ and $k>1$, $(k-1)(x-k) \geq 3$.
Therefore $|A_x| > \binom{d+x}{3}-f_2$ which means there is at least one too many absent triangles and no homology ball with this $h$-vector exists.

\section{Construction Methods}{\label{S:Const}}

In this section we present a method for constructing balls with a large variety of different $h$-vectors.
The main theorem of the section is the following:
(we address the case where the dimension $(d-1)$ is even first; the small alterations needed in the case where $(d-1)$ is odd are discussed later in the section)

\begin{thm}{\label{T:const}}
Let $(d-1)$ be even and let $(1,h_1,h_2, \ldots , h_{d-1}, 0)$ satisfy the following conditions:

$\bullet$
$(1,h_1-1,h_2-h_1,\ldots,h_{(d-3)/2} - h_{(d-5)/2},  \max \{h_{(d-1)/2} - h_{(d-3)/2},0\} )$ is an M-vector.

$\bullet$
$(1,h_1-h_{d-1},h_2-h_{d-2}, \ldots , h_{(d-1)/2} - h_{(d+1)/2})$ is an M-vector.

$\bullet$
$h_{(d+1)/2}\geq h_{(d+3)/2} \geq \cdots \geq h_{d-1}$.

\noindent Then there exists a triangulated $(d-1)$-ball with $h$-vector $(1,h_1,h_2, \ldots , h_{d-1}, 0)$.
\end{thm}

Theorem \ref{T:const} does not obtain all possible $h$-vectors of balls.
In \cite[Theorem 2]{LeeSchmidt}, Lee and Schmidt show that $h$-vectors with $h_1 \geq h_2 \geq \ldots \geq h_{d-1} \geq h_d = 0$ are the $h$-vectors of triangulated balls.
In dimensions five and higher, taking $h_i > h_{i-1}$ for $2 \leq i \leq (d/2-1)$ violates the conditions of Theorem \ref{T:const} (or the corresponding theorem for $(d-1)$ odd stated below).
Additionally, the construction of Billera and Lee in \cite[Section 6]{BilleraLee1} can be used to create balls with $h$-vector equal to any $M$-vector whose second half is all zeros, many of which can not be obtained using our construction.

In the proof of Theorem \ref{T:const} we divide a sphere into two complementary balls intersecting only along their common boundary.
The following lemma describes the relationship between the $h$-vectors of the two balls and the original sphere in this situation.

\begin{lemma}{\label{L:const}}
Let $\Delta$ be a $(d-1)$-dimensional triangulated sphere and $B \subset \Delta$ be a $(d-1)$-dimensional triangulated ball.
Let $C := (\Delta \backslash B) \cup (\partial B)$ be the complementary $(d-1)$-ball to $B$ in $\Delta$.
Then $h_i (C) = h_i (\Delta) - h_{d-i} (B)$.
\end{lemma}

\begin{proof}
Since $B$ and $C$ intersect only along $\partial B$ and $B \cup C = \Delta$ 
\begin{equation*}
f_i(B) + f_i (C) = f_i (\Delta) + f_i (\partial B)
\end{equation*}
A straightforward calculation shows
\begin{equation}{\label{E:SplitSphere}}
h_i(B) + h_i (C) = h_i (\Delta) + (h_i (\partial B) - h_{i-1} (\partial B) ) = h_i (\Delta) + g_i (\partial B)
\end{equation}
where we take $g_0 (\partial B) = h_0 (\partial B)= 1$.
From Theorem \ref{GDS} $ g_{i} (\partial B)=  h_i (B) - h_{d-i} (B)$.
Substituting this into equation (\ref{E:SplitSphere}) and simplifying yields $h_i (C) = h_i (\Delta) - h_{d-i} (B)$, as desired.
\end{proof}

Since the proof of Theorem \ref{T:const} relies heavily on the ideas in the Billera-Lee construction \cite[Section 6]{BilleraLee1} and Kalai's paper \cite{Kalai88} we review some of those concepts and notation here.

Let $d>0$ be an odd integer.
Define $F_d(n)$, a collection of $(d+1)$ subsets of $[n]$, by
$F \in F_d (n)$ if and only if $F = \cup_{j=1}^{(d+1)/2}\{ i_j, i_j+1\}$ where $i_1, \ldots, i_{(d+1)/2}$ are elements of $[n-1]$ such that $i_{j+1} > i_j +1$ for every $j$.
Let $I$ be any initial segment in $F_d(n)$ with respect to the partial order $<_p$.
Kalai showed that $B(I)$, the simplicial complex with facets the elements of $I$, is a shellable ball.
A shelling order is given by any linear order of the facets consistent with $<_p$ and the size of the restriction face of a facet $F$ is given by the number of pairs of vertices of $F$ not in their leftmost possible position, $|r(F)| = |\{ j : i_j \neq 2j-1\}|$.

Let $\F_{b,c}$ be the set of all monomials in the variables $Y_1, Y_2, \ldots ,Y_{b}$ of degree at most $c$.
Define a bijection $\alpha : F_d (n) \rightarrow \F_{n-d-1,(d+1)/2}$ by $\alpha(F) = \prod_{j=1}^{(d+1)/2}  Y_{e_j}$, where $e_j = i_j-2j+1$ is the amount that the $j$th pair of $F$ is displaced from its leftmost possible position and we take $Y_0 = 1$.
This bijection is order preserving between the partial orders $<_p$ on monomials and subsets of $[n]$.
Therefore, given an initial segment of monomials in $\F_{b,c}$ using the partial order $<_p$, the corresponding facets (under the map $\alpha^{-1}$) form a shellable ball with $|r(F)|$ equal to the degree of $\alpha(F)$ for all facets $F$.
In particular, a rev-lex initial segment of monomials gives rise to a shellable ball.
So given an $M$-vector, the image under $\alpha^{-1}$ of the corresponding compressed order ideal of monomials is a shellable ball with $h$-vector equal to the original $M$-vector (this is what was done in the Billera and Lee paper).

In the case where $d$ is even, the construction is altered by adding an additional vertex $\{0\}$ to each facet.

\begin{proof}[Proof of Theorem \ref{T:const}]
Define 
$$(1,g_1,g_2, \ldots ,g_{(d-3)/2}, g_{(d-1)/2}) := (1,h_1-1,h_2-h_1,\ldots,h_{(d-3)/2} - h_{(d-5)/2},  \max \{h_{(d-1)/2} - h_{(d-3)/2},0\} ).$$

Begin with the case $h_{(d-1)/2} - h_{(d-3)/2} \geq 0$. 
Since $(1,g_1,g_2, \ldots , g_{(d-1)/2},0,0,\ldots)$ is an $M$-vector, there is a compressed order ideal $I$ with this vector as its degree sequence.
We know $I \subseteq \F_{g_1,(d-1)/2} \subseteq \F_{g_1,(d+1)/2}$, so using the Billera-Lee method construct the $d$-ball $B(I)$ with $h$-vector $(1,g_1,g_2, \ldots , g_{(d-1)/2},0,0,\ldots)$.
Note that each facet of $B(I)$ contains the vertices $\{1\},\{2\}$.
The boundary of $B(I)$ is a $(d-1)$-sphere.
Theorem \ref{GDS}, the definition of the $g_i$, and the Dehn-Sommerville equations combine to give 
$$h(\partial B(I)) = (1,h_1,h_2,\ldots, h_{(d-1)/2},h_{(d-1)/2}, \ldots ,h_2,h_1,1)$$
Define 
\begin{equation*}
(1,G_1, \ldots, G_{(d-1)/2}) := 
(1,h_1-h_{d-1}, \ldots,  h_{(d-1)/2} - h_{(d+1)/2} )
\end{equation*}
We now construct a $(d-1)$-ball $\B$ in the sphere $\partial B(I)$ with $h (\B) = (1,G_1,G_2,\ldots, G_{(d-1)/2)})$.
By Lemma \ref{L:const}, the complementary ball $\partial B(I) \backslash (\B \backslash \partial \B)$ will be the desired $(d-1)$-ball.

The $(d-1)$-ball $\B$ uses the same correspondence $\alpha$ between facets and monomials except that the vertex names are shifted by one.
Given a monomial $m = \cup_{j=1}^{(d-1)/2} Y_{e_j}$ we define the potential corresponding facet of $\B$ by 
$(\alpha')^{-1} (m) := \{1\} \cup \left( \cup_{j=1}^{(d-1)/2} \{i_j, i_j+1\} \right)$ where $i_j = e_j+2j$ (instead of $i_j = e_j+2j-1$, as is the case for the correspondence $\alpha^{-1}$).

Next we characterize the facets of $\partial B(I)$ that will be used in the construction $\B$.
A set of $d$ vertices is a facet of $\partial B(I)$ if and only if it is in exactly one facet of $B(I)$.
Note that the only possible facet of $B(I)$ that can contain the face $(\alpha')^{-1} ( \prod_{j=1}^{(d-1)/2} Y_{e_j})$ is $\alpha^{-1}( Y_0 \cdot  \prod_{j=1}^{(d-1)/2} Y_{e_j'})$ where $e_j' = \max \{e_j-1,0\}$.
It follows that $(\alpha')^{-1} ( \prod_{j=1}^{(d-1)/2} Y_{e_j})$ is a facet of $\partial B(I)$ if and only if $\prod_{j=1}^{(d-1)/2} Y_{e_j'} \in I$.
Additionally, since all of the facets of $B(I)$ contain the vertex $\{1\}$, the face $F \backslash \{1\}$ is in $\partial B(I)$ for all facets $F$ of $B(I)$.

We now inductively build $\B$.
Let $\M_k$ be the set of degree $k$ monomials $m$ such that $(\alpha')^{-1}(m)$ is in $\B$.
By the previous paragraph, since $1 \in I$ we know that $(\alpha')^{-1}(1)$ is a facet of $\partial B(I)$.
We therefore set $\M_0 = \{1\}$.

Assume that for some $k > 0$ we have already chosen $\M_i$ for $i \leq k$ with $|\M_i| = G_i$.
Define the set $S_{k+1}$ to be $\{ Y_1 \cdot m : m \in \M_k\}$ (call these type one elements) as well as all of the monomials $\prod_{i=1}^{k+1} Y_{e_i}$ such that all of the $e_i > 1$ and $\prod_{i=1}^{k+1} Y_{e_i-1} \in I$ (these are called type two elements).
There are $G_k$ elements of the first type and $g_{k+1}$ elements of the second type giving a total of 
$$G_k + g_{k+1} = (h_k - h_{d-k}) + (h_{k+1} -h_k) = h_{k+1} - h_{d-k} \geq h_{k+1} - h_{d-(k+1)} = G_{k+1}$$
elements in $S_{k+1}$.

Select the first $G_{k+1}$ elements of $S_{k+1}$ in the rev-lex order to be the monomials in $\M_{k+1}$.
We must show that this set of monomials corresponds to an initial segment in the partial order $<_p$ and hence gives a (shellable) ball with the desired $h$-vector.
We show this inductively; given that the monomials in $\cup_{i=0}^{k} \M_i$ are an initial segment in $<_p$ we show that the monomials in $\cup_{i=0}^{k+1} \M_i$ still form an initial segment.

CLAIM 1:  Let $\tau$ be a type two element in $\M_k$.
Then $\{m \in \M_k  :  m \leq_{rl} \tau \}$ is a rev-lex initial segment in degree $k$.

Let $m$ be a degree $k$ monomial such that $m <_{rl} \tau$.
Let $m'$ be the monomial $m$ with all of the $Y_1$'s changed to $Y_2$'s (if $m$ does not contain $Y_1$ then $m'=m$).
Both $m$ and $\tau$ are degree $k$ monomials and $\tau$ contains no $Y_1$'s which implies $m' \leq_{rl} \tau$.
Since $I$ is a compressed order ideal the type two elements of $S_k$ form a compressed order ideal in the variables $Y_2, Y_3, \ldots$.
Therefore $m'$ must be a type two element of $S_k$ and $m' \in \M_k$.
However, $m <_p m'$  which by the inductive assumption implies $m \in \M_k$, proving the desired claim.

CLAIM 2:
Let $i \leq k$.  If there exists a type two element $\tau \in S_i$ such that $\tau \notin \M_i$ then $\M_i$ is a rev-lex initial segment in degree $i$.

Our proof is by induction on $i$.
The result is trivial in the base case $i=1$ where all initial segments in the order $<_p$ are also rev-lex initial segments.

Assume the claim holds for $i = l-1 \geq 1$.
Let $M$ be the rev-lex largest element of $\M_l$.
If $M$ is a type two element then we are done by the previous claim, so assume that $M$ is a type one element.
Let $m$ be a degree $l$ monomial such that $m <_{rl} M$.
We must show that $m \in \M_l$.

If $m$ does not contain the variable $Y_1$, 
then the fact that $m <_{rl} \tau$ means that $m$ is a type two element of $S_l$.
Then the fact that $m <_{rl} M$ forces $m \in \M_l$.

If $m$ contains $Y_1$ then $m/Y_1 <_{rl} M/Y_1 \leq_{rl} \tau/Y_j$ where $Y_j$ is the smallest (positive) index such that $Y_j$ is in $\tau$.
We also know that $M/Y_1 \in \M_{l-1}$ and $\tau/Y_j$ is a type two element in $S_{l-1}$.
If $\tau/Y_j \in \M_{l-1}$ then by claim one $m/Y_1 \in \M_{l-1}$.
This means $m$ is a type one element of $S_l$ which forces $m \in \M_l$.
If $\tau/Y_j \not\in \M_{l-1}$ then by the inductive hypothesis and the fact that $M/Y_1 \in \M_{l-1}$ we have $m/Y_1 \in \M_{l-1}$ which forces $m \in \M_l$, completing the proof of the claim.

CLAIM 3:  Let $\rho$ be a type two element in $\M_{k+1}$.
Then $\{m \in \M_{k+1}  :  m \leq_{rl} \rho \}$ is a rev-lex initial segment in degree $k+1$.

Our proof of claim three will be in two cases.
First consider the case where there exists a type two element $\tau \in S_k$ such that $\tau \notin \M_k$.
By claim two $\M_k$ is a rev-lex initial segment in degree $k$.
Let $N$ be the rev-lex smallest degree $k+1$ monomial not in $\M_{k+1}$.
It is sufficient to show $\rho <_{rl} N$.

If $N$ is not one of the first $G_{k+1}$ monomials of degree $k+1$ in the rev-lex order then all of the elements of $\M_{k+1}$ are rev-lex less than $N$, proving the desired claim.
We therefore assume that $N$ is one of the first $G_{k+1}$ monomials of degree $k+1$ in the rev-lex order.
If $N$ contains the variable $Y_1$ then since $\M_k$ is a rev-lex initial segment and the $G_i$ form an $M$-vector $N/Y_1 \in \M_k$.
This means $N$ is a type one element in $S_{k+1}$ and therefore $N \in \M_{k+1}$, contradicting our definition of $N$.
Thus $N$ does not contain the variable $Y_1$.
Since the type two elements of $\M_{k+1}$ are a rev-lex initial segment in $Y_2, Y_3, ...$  all of the type two elements of $\M_{k+1}$ are rev-lex less than $N$, proving the desired claim for the first case.

Now consider the case where all of the type two elements of $S_k$ are in $\M_k$.
Let $n$ be a degree $k+1$ monomial with $n <_{rl} \rho$.
We need to show that $n \in \M_{k+1}$.
If $n$ does not contain $Y_1$ then the result follows from the initial segment property of the type two elements of $S_{k+1}$.
If $n$ contains $Y_1$, then consider $n/Y_1 \leq_{rl} \rho/Y_j$ where $Y_j$ is the smallest variable in $\rho$.
Since the $g_k$ form an $M$-vector and all of the type two elements of degree $k$ are in $\M_k$ we know $\rho/Y_j$ is a type two element in $\M_k$.
Claim one then shows that $n/Y_1$ is in $\M_k$ which forces $n \in S_{k+1}$ and $n \in \M_{k+1}$.
This completes the proof of claim three.

CLAIM 4:  $\M_{k+1}$ is an initial segment of degree $k+1$ monomials in the partial order $<_p$.

To prove claim four we take any monomial $m \in \M_{k+1}$ and any degree $k+1$ monomial $m'$ with $m' <_p m$ and show that $m' \in \M_{k+1}$.
From claim three $\M_{k+1}$ consists of a rev-lex initial segment of degree $k+1$ monomials along with a (possibly empty) collection of additional type one monomials added in rev-lex order.
Since any rev-lex initial segment is also an initial segment in the partial order, we need only consider the case where $m$ is one of the type one monomials that is not part of the rev-lex initial segment in $\M_{k+1}$.
Because $m$ is a type one element $m$ contains the variable $Y_1$.
The fact that $m' <_p m$ forces $m'$ to contain $Y_1$.
Then $m'/Y_1 <_p m/Y_1$ and $m/Y_1 \in \M_k$ which combined with the inductive hypothesis implies $m'/Y_1 \in \M_k$.
This means $m'$ is a type one element of $S_{k+1}$.
Since $m' <_{rl} m$ we know $m' \in \M_{k+1}$, proving claim four.

We now show that all of the monomials in $\cup_{i=0}^{k+1} \M_i$ form an initial segment in the order $<_p$.
Given the inductive hypothesis and claim four, all we have left to show is that if $m \in \M_{k+1}$ and $m'$ is a monomial of degree less than or equal to $k$ with $m' <_p m$ then $m' \in \cup_{i=0}^{k} \M_i$.
Let $Y_j$ be the smallest variable in $m$, so that $m' <_p m/Y_j$.
By the inductive hypothesis it is sufficient to show $m/Y_j  \in \M_{k}$.

For type one elements $Y_j = Y_1$ and the claim follows from the definition of a type one element.
For type two elements we consider two cases.
For the case where there exists a type two element $\tau \in S_k$ such that $\tau \notin \M_k$ the result follows from claims two and three along with the fact that the $G_i$ form an $M$-vector.
In the case where all of the type two elements of $S_k$ are in $\M_k$ the result was already shown in the proof of claim 3.

We now address the case where $h_{(d-1)/2} - h_{(d-3)/2} < 0$.  In this case $ g_{(d-1)/2} =0$ which implies $h(\partial B(I)) = (1,h_1,h_2,\ldots, h_{(d-3)/2},h_{(d-3)/2}, h_{(d-3)/2},h_{(d-3)/2}, \ldots ,h_2,h_1,1)$.
Therefore, we alter our definition of the $G_i$.
\begin{align*}
(1,G_1, \ldots, &G_{(d-3)/2}, G_{(d-1)/2}, G_{(d+1)/2}) := \\
&(1,h_1-h_{d-1}, \ldots,  h_{(d-3)/2} - h_{(d+3)/2}, h_{(d-3)/2} - h_{(d+1)/2}, h_{(d-3)/2} - h_{(d-1)/2} )
\end{align*}
We use the same argument as above to construct the $\M_{k}$ with $k < (d-1)/2$.
Note that in this case the decreasing assumption on the $h_i$'s implies that $G_{(d-3)/2} \geq G_{(d-1)/2} \geq G_{(d+1)/2}$.
This allows us to choose $\M_{(d-1)/2}$ to be the first $G_{(d-1)/2}$ type one elements of degree $(d-1)/2$ in the rev-lex order.
Then $\cup_{k=0}^{(d-1)/2} \M_k$ is an initial segment in $<_p$, so the corresponding facets form a shellable ball with $h$-vector $(1,G_1, \ldots, G_{(d-3)/2}, G_{(d-1)/2},0,0, \ldots)$.

Let $E$ be the first $G_{(d+1)/2}$ monomials in $\M_{(d-1)/2}$ using the rev-lex order and let $m$ be a monomial such that $(\alpha')^{-1}(m) \in E$.
The degree of $m$ is $(d-1)/2$ meaning that all pairs of vertices in $m$ are shifted to the right and $\{2\} \notin m$.
However, $\{2\}$ is in every facet of $B(I)$ and since $m \in  \M_{(d-1)/2}$ we know $(\alpha')^{-1}(m)$ is contained in exactly one facet of $B(I)$.
Therefore $(\alpha')^{-1}(m) \cup \{2\}$ must be a facet of $B(I)$.
As argued above this implies that $\gamma (m) := (\alpha')^{-1}(m) \cup \{2\} \backslash \{1\}$ is in $\partial B(I)$ for all $m \in E$.

In order to get the desired $(d+1)/2$ entry of our $h$-vector, for each $m \in E$ we add the facet $\gamma(m)$ to our complementary ball.
Though the facets in this last step do not correspond to monomials using the map $\alpha'$, because of the relationship between $(\alpha')^{-1}(m)$ and $\gamma (m)$ it is straightforward to check that adding the $\gamma (m)$ to the end of the shelling in the same order as the $(\alpha')^{-1}(m)$ still gives a shellable ball with the correct $h$-vector.

\end{proof}

For the case where $(d-1)$ is odd we slightly alter the argument.
We now construct a ball with $h$-vector $(1,h_1, h_2, \ldots , h_{d-1}, 0)$ that satisfies the following conditions:

$\bullet$
$(1,h_1-1,h_2-h_1,\ldots, h_{d/2-1} - h_{d/2-2}, \max \{ h_{d/2}-h_{d/2 - 1}, 0 \})$ is an M-vector.

$\bullet$
$(1,h_1-h_{d-1},h_2-h_{d-2}, \ldots , h_{d/2-1} - h_{d/2+1})$ is an M-vector.

$\bullet$
$h_{d/2}\geq h_{d/2+1} \geq \cdots \geq h_{d-1}$.

With appropriate changes in notation and parity, the same argument as above can be used to prove this result.

It is also possible to shell the complementary ball constructed in Theorem \ref{T:const}.
The highly technical proof of this fact is included in appendix.
This proof is independent of the work in proof of Theorem \ref{T:const}.

\section{Consequences of the Construction}{\label{S:Sum}}

As noted in the previous section, the conditions of Theorem \ref{T:const} are not in general necessary for the existence of a ball with a given $h$-vector.
However, in dimensions three and four it is straightforward to check that the conditions of Conjecture \ref{gconditions} imply the conditions of Theorem \ref{T:const}.
Since we have already shown the necessity of the conditions of Conjecture \ref{gconditions} in dimensions three and four, these conditions give a complete characterization of the $h$-vectors of three and four dimensional balls.
As mentioned previously, this result was first obtained by Lee and Schmidt in \cite{LeeSchmidt}.

Starting in dimension five we know that the conditions of Conjecture \ref{gconditions} are no longer sufficient and we also know that the conditions of Theorem \ref{T:const} are no longer necessary.
In particular, our construction can not create any five-balls with $h_2 < h_1$, even though many such balls exist.
Given any five-ball we can attach to it a single five-simplex by gluing along a single co-dimension one face of the boundary of each ball.
This process adds one new vertex and one new facet to the original ball.
As described in section \ref{S:glue}, this increases the $h_1$ value of the original ball by one without changing any of the other entries of the $h$-vector.
Repeating this process we can create many different balls with $h_2 < h_1$.

While there exist balls with $h_2 < h_1$ that do not arise from adding vertices to other balls as described above, we have so far been unable to find any five-ball whose $h$-vector can not be realized by adding vertices to a ball constructed using Theorem \ref{T:const}.
In fact, using the methods of section \ref{S:nec} it can be shown that many of the `small' $h$-vectors that cannot be obtained by adding vertices to balls constructed using Theorem \ref{T:const} cannot be the $h$-vectors of five-balls.
We therefore make the following conjecture.

\begin{conj}{\label{C:5balls}}
A vector $h = (1,h_1,h_2, h_3,h_4,h_5,0)$ is the $h$-vector of a five-ball if any only if there exists some integer $m > 0$ such that  $h = (1,h_1-m,h_2, h_3,h_4,h_5,0)$ satisfies the conditions of Theorem \ref{T:const}.
\end{conj}

If any two balls satisfying the condition in Conjecture \ref{C:5balls} are joined along a single co-dimension one face, the $h$-vector of the resulting ball still satisfies the conditions of the conjecture.
However, it is not true that the conditions of Theorem \ref{T:const} give all of the $h$-vectors of balls that cannot be split along some co-dimension one face (i.e. all the balls $\Delta$ with $\beta_{n-d,n-d+1} (k(\Delta)) = 0$).
As an example, combining the ball with $h$-vector (1,3,6,10,5,3,0) formed from the construction of Theorem \ref{T:const} with the shellable ball with $h$-vector $(1,2,0,0,0,0,0)$ by gluing along two boundary faces gives a ball with $h$-vector (1,5,7,10,5,3,0) but no co-dimension one face along which to split.

Beyond dimension five we know that conjecture \ref{gconditions} does not a give a description of the $h$-vectors of balls but we do not have any conjecture to replace it.
Determining even just a conjectural description of the $h$-vectors of these higher dimensional balls remains an interesting open problem.

\appendix
\section{Shellable Ball Construction}

In this appendix we strengthen the result of Theorem \ref{T:const} by showing that there exists a \emph{shellable} ball with the desired $h$-vector.  We begin with the case where $d-1$ is even, the other case is handled similarly.
Define 
$$(1,g_1,g_2, \ldots ,g_{(d-3)/2}, g_{(d-1)/2}) := (1,h_1-1,h_2-h_1,\ldots,h_{(d-3)/2} - h_{(d-5)/2},  \max \{h_{(d-1)/2} - h_{(d-3)/2},0\} ).$$
Let $J$ be the compressed order ideal of monomials with degree sequence $(1,g_1,g_2, \ldots , g_{(d-1)/2},0,0,\ldots)$.
Following the Billera-Lee construction, build the $d$-ball $B(J)$ with $h$-vector $(1,g_1,g_2, \ldots , g_{(d-1)/2},0,0,\ldots)$.
We show that in the boundary sphere $\partial B(J)$ there is a shellable ball with the desired $h$-vector.

Begin with the case $h_{(d-1)/2} - h_{(d-3)/2} \geq 0$. 
Define 
\begin{equation*}
(1,G_1, \ldots, G_{(d-1)/2}) := 
(1,h_1-h_{d-1}, \ldots,  h_{(d-1)/2} - h_{(d+1)/2} )
\end{equation*}
Our first goal is to construct a shellable $(d-1)$-ball $\C(J)$ in $\partial B(J)$ with 
$$h(\C(J)) =( 1,h_1, h_2, \ldots , h_{(d-1)/2},0 ,0 , \ldots)$$

Let $I$ be a rev-lex initial segment of the monomials of $J$.
Note that $I$ is a compressed order ideal.
Let $(c_0=1 ,c_1, \ldots , c_{(d-1)/2})$ be the degree sequence of $I$.
The $d$-ball $B(I)$ has $h$-vector $(1,c_1, \ldots , c_{(d-1)/2},0,0, \cdots , 0)$.
Using induction on the number of monomials in $I$ we construct a shellable $(d-1)$-ball $\C(I)$ in $\partial B(I)$ with $h$-vector
$$h(\C(I)) = (1, 1 + c_1, 1+ c_1 + c_2, \cdots , \sum_{i=0}^{(d-1)/2} c_i )$$
The case $I=J$ constructs the desired ball $\C(J)$.

We begin by describing the facets of $\C(I)$.
Let $m$ be any monomial in $I$.
Write $\alpha^{-1}(m)$ as 
$$\alpha^{-1}(m) =  \cup_{j=1}^{(d+1)/2} \{i_j,i_{j+1}\}$$
For $k \in \{1,2, \ldots , (d+1)/2 \}$ and $m \in J$, the sets in $ \partial B(I) \cap \alpha^{-1}(m)$ that have the form 
\begin{equation}{\label{facetform}}
\left( \cup_{1 \leq j \leq (d+1)/2, j \neq k} \{i_j,i_{j+1}\} \right) \cup \{ i_{k+1} \}
\end{equation}
will be facets of $\C(I)$.
Note that for any face $F$ described in (\ref{facetform}) there are exactly two monomials $m,m'$ such that $\alpha^{-1} (m)$ and $\alpha^{-1} (m')$ contain $F$.
Let $m <_{rl} m'$.
The face $F$ is contained in $\C(I)$ for exactly those ideals $I$ that contain $m$ but not $m'$.

Next we describe the restriction face for each facet of $\C(I)$.
We can write any facet $F$ of $\C(I)$ as $F = \cup_{j=1}^{d} \{ p_j \}$ where $0 \leq p_1 < p_2 < \cdots < p_d$.
Let $l(F)$ be the size of the left endset of $F$, $l(F) := |\{ k : p_k = k\}|$.
The restriction face $r(F)$ in our shelling will be $\{ p_k : k = l(F) + 2n, n \in \N\}$.
Equivalently, starting after the left endset, every second vertex of $F$ is in $r(F)$.

Finally we describe the order for the shelling of $\C(I)$.
To do this we write each facet $F$ of $\C(I)$ as a disjoint union
$$F = L(F) \cup \left (\cup_{i=1}^{c_F} K_i(F) \right)$$
Here $L(F)$ is the left endset of $F$, $L(F) := \{i\}_{i=1}^{l(F)}$.
Note that $L(F)$ may be empty.
Each set $K_i(F)$ is a non-empty contiguous set of vertices $\{a,a+1, \ldots , b-1, b\}$ such that $a-1$ and $b+1$ are not in $F$.
The $K_i(F)$ are chosen so that for $i<j$ any element of $K_i(F)$ is less than any element of $K_j(F)$.

Let $F$ and $G$ be distinct facets of $\C(I)$.

If $|L(F)| > |L(G)|$ then we place $F$ before $G$ in our shelling order.
This is equivalent to ordering the facets by the size of their restriction faces in increasing order.

If $|L(F)| = |L(G)|$, then let $a$ and $b$ be the smallest elements of $K_1(F)$ and $K_1(G)$ respectively.
If $a<b$ then we place $F$ before $G$ in the shelling.

If $a=b$, $|K_1(F)|$ is odd, and $|K_1(G)|$ is even then we place $F$ before $G$ in the shelling.

If both $|K_1(F)|$ and $|K_1(G)|$ are odd and $|K_1(F)| < |K_1(G)|$ then place $F$ before $G$ in the shelling.
If $|K_1(F)| = |K_1(G)|$ is odd then let $e$ be the largest vertex that is in either $F$ or $G$ but not both.
Place the facet that contains $e$ earlier in the shelling.

If both $|K_1(F)|$ and $|K_1(G)|$ are even and $|K_1(F)| < |K_1(G)|$ then place $G$ before $F$ in the shelling.
If $|K_1(F)| = |K_1(G)|$ is even then repeat the above steps on $K_2(F)$ and $K_2(G)$ rather than the $K_1$'s.
Continuing in this manner, since $F$ and $G$ are distinct there will be some $i$ such that $K_i(F) \neq K_i(G)$ which will result in some way to order $F$ and $G$.

Now we inductively show that the above facets, restriction faces, and ordering give a shelling order for $\C(I)$.
For the base case $I = \{ 1 \}$ we have $B(I) =  \alpha^{-1} (1) = \{ i \}_{i=1}^{d+1}$.
The above described ordering gives a shelling $\{F_i\}_{i=1}^{(d+1)/2}$ where 
\begin{equation*}
F_i = \cup_{1 \leq j \leq d+1, j \neq d-2i+2} \{j\}
\end{equation*}

It is straightforward to check that this is a shelling order with restriction faces $r(F_i)$ as defined above and the desired $h$-vector.

We now consider the step where we add a new monomial to our ideal and a corresponding new facet to our $d$-ball.
Let $I$ be the previous ideal and let $I' = I \cup \{m\}$ be our new ideal.
So  $\alpha^{-1}(m)$ is the new face we add to get the ball $B(I')$.
Let $s(m)$ be the degree of $m$.
We write $\alpha^{-1}(m)$ as a disjoint union
\begin{equation}{\label{alphainverse}}
\alpha^{-1}(m) = \left( \cup_{i=1}^{d+1-2 \cdot s(m)} \{i\} \right) \cup \left( \cup_{j=1}^{s(m)} \{i_j,i_j+1\} \right)
\end{equation}
where $i_{j+1} > i_j +1$ for every $j$.
Note that $s(m)$ is the number of shifted pairs of $\alpha^{-1}(m)$.

Consider the facets $F$ of $\C(I)$ and $C(I')$ that have $|L(F)| = d+1-2 \cdot s(m)$.
The faces
\begin{equation}{\label{removedfaces}}
\left( \cup_{i=1}^{d+1-2 \cdot s(m)} \{i\} \right) \cup \left( \cup_{1 \leq j \leq s(m), j \neq k} \{i_j,i_j+1\} \right) \cup \{i_k\}
\end{equation}
for  $k=1,2, \ldots , s(m)$ are in $\C(I)$, but the addition of $\alpha^{-1} (m)$ to $B(I)$ causes these faces to not be in $\partial B(I')$ and hence also to not be in $\C(I')$.
The faces with left endset of size $d+1-2 \cdot s(m)$ that are in $\C(I')$ but not $\C(I)$ are 
$$\left( \cup_{i=1}^{d+1-2 \cdot s(m)} \{i\} \right) \cup \left( \cup_{1 \leq j \leq s(m), j \neq k} \{i_j,i_j+1\} \right) \cup \{i_k+1\}$$
where $k=1,2, \ldots , s(m)$.

We show that the facets $F$ of $\C(I')$ with  $|L(F)| \geq d+1-2 \cdot s(m)$ ordered as described above give a shelling order with the claimed restriction faces. 

No facets $F$ with $|L(F)| > d+1-2 \cdot s(m)$ are added or removed by moving from $\C(I)$ to $C(I')$.
Therefore, we need only check the result for the facets of $\C(I')$ with left endset of size $d+1-2 \cdot s(m)$.

CLAIM 1:
Let $F$ be a facet in $\C(I') \backslash \C(I)$ with $|L(F)| = d+1-2 \cdot s(m)$.
Then $r(F)$ is the unique minimal new face of $F$ in the above order.

To show that $r(F)$ is the unique minimal new face we show that it is a new face and that for each vertex $v \in r(F)$ the face $F - \{v\}$ is in some previous facet of our shelling order.

Each restriction face $r(F)$ for $F \in \C(I') \backslash \C(I)$ is equal to $r(G)$ for some $G \in \C(I) \backslash \C(I')$.
Therefore, for $F \in \C(I') \backslash \C(I)$, $r(F)$ can not be contained in some face $H \in \C(I')$ with $|L(H)| > d+1-2 \cdot s(m)$.
For $H \in \C(I')$ with $|L(H)| = d+1-2 \cdot s(m)$, one can check exhaustively that if $H$ contains $r(F)$ then $H$ is after $F$ in our ordering.
This exhaustive check uses the fact that exactly one of the monomials $n$ such that $\alpha^{-1} (n)$ contains $H$ must be rev-lex less than $m$.
This completes the proof that $r(F)$ is a new face.

From the inductive assumption we know that $\C(I)$ is shellable with the above described shelling order and restriction faces.
For each vertex $v \in r(F)$, the face $F - \{v\}$ is either in one of the facets of $\C(I') \backslash \C(I)$ that is before $F$ in the shelling order or is equal to $G - \{w\}$ where $G$ is one of the facets in $\C(I) \backslash \C(I')$ and $w \in r(G)$.
These two cases are distinguished by whether the vertex $v$ comes after (or in) the one contiguous set in $F$ of odd length or before the odd contiguous set.

In the former case we immediately have that $F-\{v\}$ is in some previous facet of our shelling order.
In the latter case we know that there is some facet $G'$ in $\C(I)$ before $G$ in the shelling order such that $G - \{w\} \subseteq G'$.
It is straightforward to check that the structure of the facets of $C(I) \backslash \C(I')$ forces $G' \in C(I')$.
The fact the all of the facets of $\C(I') \backslash \C(I)$ come later in our ordering than the facets of $\C(I) \backslash \C(I')$ shows that $F-\{v\}$ is in some previous facet of our shelling order, finishing claim one.

CLAIM 2:
Let $F$ be a facet in $\C(I') \cap \C(I)$ with $|L(F)| = d+1-2 \cdot s(m)$.
Then $r(F)$ is the unique minimal new face of $F$ in the above order.

By the inductive hypothesis we know that $r(F)$ is a new face in the shelling order on $\C(I)$.
Note that all of the faces of $\C(I') \backslash \C(I)$ contain the face $Q := \cup_{j=1}^{s(m)} \{i_j+1\} $.
The face $Q$ is not contained in any face of $\C(I)$ because it is the restriction face of $\alpha^{-1} (m)$ in $B(I')$ and hence is not a face of $B(I)$.
Therefore $r(F) \notin \C(I') \backslash \C(I)$ and $r(F)$ is a new face in $\C(I')$.

To complete the proof of claim two we must show that for each vertex $v \in r(F)$ the face $F - \{v\}$ is in some facet of $\C(I')$ that is before $F$ in our ordering.
Let $G':= F - \{v\}$.
Choose the first facet $G$ in the ordering on $\C(I)$ such that $G' \subseteq G$.
By the inductive hypothesis we know that such a $G$ exists and that $G$ is before $F$ in the ordering on $\C(I)$.
If $G \in \C(I')$ we are done, so assume $G \in \C(I) \backslash \C(I')$.
Write $G$ in the form (\ref{removedfaces}) above.

If $G'$ is obtained from $G$ by removing $\{i_k\}$ then it is impossible to add any vertex $w$ to $G'$ such that $w \in r(G' \cup \{w\})$ since $w$ can not be an even numbered vertex after the first gap in $G' \cup \{w\}$.
Therefore, $G'$ can not be equal to $ F - \{v\}$, and this case does not occur.

If $G'$ is obtained from $G$ by removing $\{i_j+1\}$ for $j < k$ or $\{i_j\}$ for $j > k$, then $G'$ does not contain $r(G)$.
Therefore $G'$ is contained in some face of $\C(I)$ before $G$ in our ordering, contradicting our minimality assumption on $G$.

If $G'$ is obtained from $G$ by removing $\{i_j+1\}$ for $j > k$ then for every vertex $w$ we can show that $G' \cup \{w\}$ is not equal to $F$ for (at least) one of the following reasons:
\begin{enumerate}
\item  $G' \cup \{w\}$ is not of the form (\ref{facetform}).
\item  $G' \cup \{w\}$ is contained in $\alpha^{-1} (n)$ and $\alpha^{-1} (n')$ for two distinct monomials $n,n'$ both rev-lex less than $m$.
\item  $G' \cup \{w\}$ is in $\C(I') \backslash \C(I)$.
\end{enumerate}

If $G'$ is obtained from $G$ by removing $\{i_j\}$ for $j < k$ then for every vertex $w$ either $G' \cup \{w\}$ is not equal to $F$ for one of the above reasons or there is a facet $H \in \C(I') \backslash \C(I)$ such that $G' \subseteq H$ and $H$ is before $G' \cup \{w\}$ in our ordering.
This completes claim 2.

Now we consider the latter part of our shelling consisting of faces $F$ such that $|L(F)| < d+1-2 \cdot s(m)$.
For every face $F$ in $\C(I) \backslash \C(I')$, $|L(F)| = d+1-2 \cdot s(m)$.
Thus no faces are removed in the later part of the shelling.
For each even $k$ such that $0 \leq k < d+1-2 \cdot s(m)$ there is one face $F_k$ in $\C(I') \backslash \C(I)$ with $|L(F_k)| = k$.
Using the notation of (\ref{alphainverse}) for $\alpha^{-1}(m)$ this new face is given by
\begin{equation*}
F_k = \left( \cup_{i=1}^{d+1-2 \cdot s(m)} \{i\} \right) \backslash \{k+1\} \cup \left( \cup_{j=1}^{s(m)} \{i_j,i_j+1\} \right)
\end{equation*}

CLAIM 3:
Let $F_k$ be a facet in $\C(I') \backslash \C(I)$ with $|L(F_k)| < d+1-2 \cdot s(m)$.
Then $r(F_k)$ is the unique minimal new face of $F_k$ in the above order.

Note that 
$$r(F_k) = \left( \cup_{i=k/2}^{d/2- s(m)-3/2} \{2i+3\} \right) \cup \left( \cup_{j=1}^{s(m)} \{i_j\} \right)$$

Fix $k$ and let $G$ be a face of $\C(I')$ such that $r(F_k)  \subseteq G$.
If $|L(G)| > k$ then because $I$ is a compressed order ideal $G$ would already be in two facets of $B(I')$ and hence would not be in $ \partial B(I')$.
If $|L(G)| = k$ then by checking all possible choices for $G$ we see that $G$ must be after $F_k$ in our ordering.
If $|L(G)| < k$ then by the definition of our ordering $G$ is after $F_k$.
Therefore, $r(F_k)$ is a new face.

Any face of $F_k$ that does not contain $\cup_{j=1}^{s(m)} \{i_j\}$ is contained in one of the facets of $\C(I') \backslash \C(I)$ with $|L(F_k)| = d+1-2 \cdot s(m)$.
Therefore $r(F_k)$ is the unique minimal new face of $F_k$, as desired.

CLAIM 4:
Let $F$ be a facet in $\C(I') \cap \C(I)$ with $|L(F)| < d+1-2 \cdot s(m)$.
Then $r(F)$ is the unique minimal new face of $F$ in the above order.

The fact that $r(F)$ is a new face is proved as in claim two.

Now let $v \in r(F)$ and consider the face $F' := F - \{v\}$.
We need to show that $F' \subseteq G$ for some $G \in \C(I')$ such that $G$ is before $F$ in our ordering.
By the inductive hypothesis we know that there is an $H \in \C(I)$ such that $F' \subset H$ and $H$ is before $F$ in our ordering.
The result is trivial if $H \in \C(I')$, so assume $H \in \C(I) \backslash C(I')$.
If $H \backslash F'$ were an even number in $L(H)$ then $|L(F)|$ would be odd, contradicting the form of the elements of $\C(I)$.
If $H \backslash F' = \{k\}$ is an odd number in $L(H)$ then the facet $H' \in \C(I') \backslash C(I)$ with $|L(H')| = k-1$ contains $F'$.
In our ordering $H'$ comes before $F$ because of the parity of each facet's leftmost contiguous set (that is not the left endset).
When $H \backslash F'$ is not in $L(H)$ there is a face $H' \in \C(I') \backslash C(I)$ with $|L(H')| = d+1-2 \cdot s(m)$ such that $F' \subset H'$.
By our ordering $H'$ is before $F$, proving the claim in this case.
Therefore $r(F)$ is the unique minimal new face of $F$, completing claim four.

We have now shown that $C(I')$ is shellable with the claimed order and restriction faces.
Using this shelling, it is straightforward to check inductively that $C(I')$ has the desired $h$-vector.

\vspace{6pt}

Next we add additional facets to our shelling of $\C(J)$ to complete the second half of the desired $h$-vector.

Given a facet $\alpha^{-1}(m)$ of $B(J)$ in the form (\ref{alphainverse}), the co-dimension one faces obtained by removing $\{k\}$ where $k$ is even and $1 < k \leq d+1-2 \cdot s(m)$ are facets of $\partial B(J)$.
These are the facets we add during the second half of the shelling.
For a fixed $k$, call the collection of all such facets $A_k'$.
For each facet $F \in A_k'$ used in the second half of our shelling we will have
$$
r(F) = \left( \cup_{i=1}^{k-1}  \{i\} \right) 
\cup \left( \cup_{i=k/2}^{(d-1)/2- s(m) } \{2i+1\} \right) 
\cup   \left( \cup_{j=1}^{s(m)} \{i_j\} \right)
$$

Given a facet $\alpha^{-1}(m)$ of $B(J)$ we will not add every co-dimension one face of $\alpha^{-1}(m)$ of the form described above to our shelling.
However, if we choose not to add some co-dimension one face $F$ of $\alpha^{-1} (m)$ to our shelling, then any other co-dimension of face $F'$ of $\alpha^{-1} (m)$ with $|r(F')| > |r(F)|$ also can not be added to the shelling.
Additionally, for $m' \in J$ rev-lex before $m$ we will not allow the addition of any co-dimension one face $G$ of $\alpha^{-1} (m')$ such that $|r(G)| \geq |r(F)|$.

For each facet $F$ we are adding to the second half of our shelling there is a unique monomial $m \in J$ such that $F \in \alpha^{-1} (m)$.
We order our facets by the rev-lex order on the corresponding monomials in $J$ with the facets corresponding to the rev-lex larger monomials coming earlier in our ordering.
We order facets that have the same corresponding monomials by increasing size of their restriction faces.
We show that under these conditions we get a shelling order than extends the shelling of $\C(J)$.

Let $F \subset \alpha^{-1}(m)$ be a face added in the second half of the shelling.
We first show that $r(F)$ is a new face.

The only faces of the form (\ref{facetform}) that can contain $r(F)$ are contained in $\alpha^{-1} (n)$ and $\alpha^{-1} (n')$ for distinct monomials $n,n'$ both rev-lex before $m$.
Therefore $r(F)$ is not contained in any monomial in the first half of the shelling.
Also note that the only facets of $B(J)$ that contain $r(F)$ are equal to $\alpha^{-1} (m')$ for a monomial $m'$ that is rev-lex less than equal to $m$.
Therefore no facet in the second half of the shelling but before $F$ in our ordering contains $r(F)$.
Thus $r(F)$ is a new face.

Now let $v \in r(F)$.
We show that $F' := F - \{v\}$ is contained in some facet before $F$ in our shelling.

If $v < k$ and $v$ is even then the face $F' \cup \{k\}$ is a co-dimension one face of $\alpha^{-1}(m)$ and is in our shelling by the conditions on restriction face size.
This face gives the desired co-dimension one intersection.

If $v < k $ and $v$ is odd we have two cases
\begin{enumerate}
\item  The facet $F' \cup \{k\}$ is a facet in our shelling, proving the desired result.

\item  If $F' \cup \{k\}$ is not in $\C(J)$ then the facet $F' \cup \{k\} \cup \{w\}$ is in $B(J)$ where $w$ is the smallest number that is greater than $d+1- 2 \cdot s(m)$ and not in $F$.
Therefore the face $F' \cup \{w\}$ is in $B(J)$.
If $F' \cup \{w\}$ is in $\C(J)$ then we are done.
If not we let $w'$ be the smallest number greater than $w$ that is not in $F$.
We must have $F' \cup \{w\} \cup \{w'\}$ in $B(J)$.
If $F' \cup \{w'\}$ is in $\C(J)$ we are done.
If not we repeat the above process until we reach a facet that is in $\C(J)$ and contains $F'$.
\end{enumerate}

If $v \in r(F)$ and $v > k$ we have two cases.

\begin{enumerate}
\item   The facet $F' \cup \{k\}$ is a facet in our shelling, and we are done with this step.

\item If $F' \cup \{k\}$ is not in $\C(J)$ then the face $F' \cup \{k\} \cup \{w\}$ must be in $B(J)$ where $w$ is the smallest number that is greater than $v$ and is not in $F$.
The facet $F' \cup \{k\} \cup \{w\}$ is after $\alpha^{-1} (m)$ in the shelling order on $B(J)$ induced by the rev-lex ordering on $J$.
Hence, by our conditions on which facets must be added to the second half of our shelling, $F' \cup \{w\}$ must be a facet in the second half of the shelling and before $F$ in our ordering, completing the desired claim in this case.
\end{enumerate}

Finally we show that we get all of the claimed $h$-vectors using this construction.
Let $A_k'$ be as above.
For each face $F \in A_k'$ we know $F \cup \{k\}$ is an element of $B(J)$ and corresponds to a monomial in $J$.
Order the elements of $A_k'$ by the rev-lex order on these corresponding monomials with the rev-lex largest monomials first.
Let $A_k$ be the first $h_{(d+k-1)/2}$ elements of $A_k'$ using this ordering.

We must show that for any facet $\alpha^{-1}(m)$ of $B(J)$ and any even $k \geq 2$ if the face $\alpha^{-1}(m) \backslash \{k+2\}$ is in $A_{k+2}$ then $\alpha^{-1}(m) \backslash \{k\}$ is in $A_k$.

If there are no monomials $n \in J$ such that the degree of $n$ is $(d-k+1)/2$ and $n$ is after $m$ in the rev lex order, then the number elements in $A_{k}$ and $A_{k+2}$ that are in facets of $B(J)$ corresponding to monomials rev-lex larger than $m$ are equal.
By the non-increasing condition on the $h_i$ we have the desired result.

If there is a monomial $n \in J$ of degree $(d-k+1)/2$  with $n >_{rl} m$ then $J$ contains all of the monomials $n'$ such that the degree of $n'$ is less than or equal to $(d-k+1)/2$ and $n' <_{rl} m$.
The entries in the vector $(1,h_1-h_{d-1},h_2-h_{d-2}, \ldots , h_{(d-1)/2} - h_{(d+1)/2})$ are how many elements of each $A_k'$ are not selected for $A_k$.
Since this is an $M$-vector and for a monomial $N \in J$ the face $\alpha^{-1} (N)$ contains an element of $A_{k}'$ if and only if $N$ has degree less than or equal to $(d-k+1)/2$ the desired result is a consequence of the following claim

CLAIM 5:
Let $(1,a_1, \ldots , a_r)$ be an $M$-vector and define $b_k := 1 + a_1 + \cdots + a_k$.
Then $b_k^{<k>} \geq b_{k+1}$.

Since a truncation of an $M$-vector is still an $M$-vector it is sufficient to prove the claim for $k=r-1$.
Let $K$ be the compressed order ideal with $M$-vector $(1,a_1, \ldots , a_r)$.
Let $K'$ be the elements of $K$ of degree at most $r-1$.
Note that $|K'| = b_{r-1}$.
For $m \in K'$ write $m = \prod_{j=1}^{r-1} X_{i_j}$ where we take $X_0 = 1$.
Define $m^+ := \prod_{j=1}^{r-1} X_{i_j+1}$

Let $S' = \{m^+ : m \in K'\}$.
It is straightforward to check that for any $m \in S'$ and any degree $r$ monomial $m'$ with $m' <_{rl} m$ we have $m' \in S'$.
Hence $S'$ is the rev-lex first $b_{r-1}$ monomials of degree $r-1$.

For $m \in K$ write $m = \prod_{j=1}^{r} X_{i_j}$ where we take $X_0 = 1$.
Let $m^+ = \prod_{j=1}^{r} X_{i_j+1}$ and define $S = \{ m^+ : m \in K \}$.
By the same argument as for $S'$, the set $S$ is the rev-lex first $b_{r}$ monomials of degree $r$.

It is easy to check that for each monomial $n \in S$, dividing $n$ by any variable in $n$ gives a monomial in $S'$.
Therefore $b_r \leq  b_{r-1}^{<r-1>}$, as desired.
\vspace{6pt}

We now consider the case where $h_{(d-1)/2} - h_{(d-3)/2} < 0$.
In this case $ g_{(d-1)/2} =0$ which implies $h(\partial B(I)) = (1,h_1,h_2,\ldots, h_{(d-3)/2},h_{(d-3)/2}, h_{(d-3)/2},h_{(d-3)/2}, \ldots ,h_2,h_1,1)$.
Therefore, we alter our definition of the $G_i$.
\begin{align*}
(1,G_1, \ldots, &G_{(d-3)/2}, G_{(d-1)/2}, G_{(d+1)/2}) := \\
&(1,h_1-h_{d-1}, \ldots,  h_{(d-3)/2} - h_{(d+3)/2}, h_{(d-3)/2} - h_{(d+1)/2}, h_{(d-3)/2} - h_{(d-1)/2} )
\end{align*}
Let $J$ be the compressed order ideal with degree sequence $(1,g_1,g_2, \ldots , g_{(d-3)/2})$.
Using the same argument as for the previous case we build a shellable $(d-1)$-ball $\C(J)$ with $h$-vector $ (1,h_1,h_2,\ldots, h_{(d-3)/2},0, \ldots , 0)$.

To obtain a ball with the desired $h$-vector we make a slight alteration to the second half of the shelling described in the previous case.
Note that for each monomial $m \in J$ the face $\alpha^{-1}(m) \backslash \{1\}$ is in $\partial B(J)$.
Call the collection of all such facets $A_1'$.
Order the elements of $A_1'$ by the rev-lex order on the corresponding monomials of $J$ with the facets corresponding to the rev-lex largest monomials first.
Note that 
$$|A_1'| = 1 + g_1 + g_2 + \cdots + g_{(d-3)/2} = h_{(d-3)/2} > h_{(d-1)/2}$$
We may therefore define $A_1$ to be the first $h_{(d-1)/2}$ elements of $A_1'$.
For $k>1$ define the $A_k$ as in the previous case.

For any facet $F \in A_1$ and corresponding monomial $m \in J$ with $F =\alpha^{-1} (M) \backslash \{ 1\}$ define 
$$
r(F) = \left( \cup_{i=1}^{(d-1)/2- s(m) } \{2i+1\} \right) 
\cup   \left( \cup_{j=1}^{s(m)} \{i_j\} \right)
$$
Then arguing as above, adding the faces in each of the $A_i$ to the the ball $\C(J)$ results in a shellable ball with the desired $h$-vector.

\section*{Acknowledgements}
The author would like to thank Ed Swartz for his valuable advice and support.

\bibliographystyle{plain}
\bibliography{Ballsbib}

\begin{thebibliography}{10}

\bibitem{Barnette73}
David Barnette.
\newblock A proof of the lower bound conjecture for convex polytopes.
\newblock {\em Pacific J. Math.}, 46:349--354, 1973.

\bibitem{BilleraLee2}
Louis~J. Billera and Carl~W. Lee.
\newblock The numbers of faces of polytope pairs and unbounded polyhedra.
\newblock {\em European J. Combin.}, 2(4):307--322, 1981.

\bibitem{BilleraLee1}
Louis~J. Billera and Carl~W. Lee.
\newblock A proof of the sufficiency of {M}c{M}ullen's conditions for
  {$f$}-vectors of simplicial convex polytopes.
\newblock {\em J. Combin. Theory Ser. A}, 31(3):237--255, 1981.

\bibitem{EliahouKervaire90}
Shalom Eliahou and Michel Kervaire.
\newblock Minimal resolutions of some monomial ideals.
\newblock {\em J. Algebra}, 129(1):1--25, 1990.

\bibitem{Hibi94}
Takayuki Hibi.
\newblock Cohen-{M}acaulay types of {C}ohen-{M}acaulay complexes.
\newblock {\em J. Algebra}, 168(3):780--797, 1994.

\bibitem{Hochster}
Melvin Hochster.
\newblock Cohen-{M}acaulay rings, combinatorics, and simplicial complexes.
\newblock In {\em Ring theory, {II} ({P}roc. {S}econd {C}onf., {U}niv.
  {O}klahoma, {N}orman, {O}kla., 1975)}, pages 171--223. Lecture Notes in Pure
  and Appl. Math., Vol. 26. Dekker, New York, 1977.

\bibitem{Kalai88}
Gil Kalai.
\newblock Many triangulated spheres.
\newblock {\em Discrete Comput. Geom.}, 3(1):1--14, 1988.

\bibitem{Klee64}
Victor Klee.
\newblock A combinatorial analogue of {P}oincar\'e's duality theorem.
\newblock {\em Canad. J. Math.}, 16:517--531, 1964.

\bibitem{LeeSchmidt}
Carl~W. Lee and Laura Schmidt.
\newblock On the numbers of faces of low-dimensional regular triangulations and
  shellable balls.
\newblock Preprint, 2009.

\bibitem{Macaulay}
F.~S. Macaulay.
\newblock Some properties of enumeration in the theory of modular systems.
\newblock {\em Proc. London Math. Soc.}, 25:531--555, 1927.

\bibitem{Macdonald71}
I.~G. Macdonald.
\newblock Polynomials associated with finite cell-complexes.
\newblock {\em J. London Math. Soc. (2)}, 4:181--192, 1971.

\bibitem{Mitchell90}
W.~J.~R. Mitchell.
\newblock Defining the boundary of a homology manifold.
\newblock {\em Proc. Amer. Math. Soc.}, 110(2):509--513, 1990.

\bibitem{Nevo08}
Eran Nevo.
\newblock Rigidity and the lower bound theorem for doubly {C}ohen-{M}acaulay
  complexes.
\newblock {\em Discrete Comput. Geom.}, 39(1-3):411--418, 2008.

\bibitem{Novik09}
Isabella Novik and Ed~Swartz.
\newblock Applications of {K}lee's {D}ehn-{S}ommerville relations.
\newblock {\em Discrete Comput. Geom.}, 42(2):261--276, 2009.

\bibitem{Peeva}
Irena Peeva.
\newblock Consecutive cancellations in {B}etti numbers.
\newblock {\em Proc. Amer. Math. Soc.}, 132(12):3503--3507 (electronic), 2004.

\bibitem{Reisner}
Gerald~Allen Reisner.
\newblock Cohen-{M}acaulay quotients of polynomial rings.
\newblock {\em Advances in Math.}, 21(1):30--49, 1976.

\bibitem{Stanleyg}
Richard~P. Stanley.
\newblock The number of faces of a simplicial convex polytope.
\newblock {\em Adv. in Math.}, 35(3):236--238, 1980.

\bibitem{StanleyCCA}
Richard~P. Stanley.
\newblock {\em Combinatorics and commutative algebra}, volume~41 of {\em
  Progress in Mathematics}.
\newblock Birkh\"auser Boston Inc., Boston, MA, second edition, 1996.

\bibitem{Swartz09}
Ed~Swartz.
\newblock Face enumeration---from spheres to manifolds.
\newblock {\em J. Eur. Math. Soc. (JEMS)}, 11(3):449--485, 2009.

\end{thebibliography}

\end{document}